\documentclass[11pt, letterpaper]{amsart}
\usepackage[utf8]{inputenc}
\usepackage{enumerate, xcolor, graphicx}
\definecolor{Magenta}{cmyk}{0,1,0,0}
\definecolor{dgreen}{rgb}{0,0.6,0.}

\usepackage{academicons}

\usepackage{amssymb,amsthm,amsmath,amsfonts}
\usepackage{diagbox}

\usepackage{mathrsfs, mathtools, mathdots}
\usepackage{paralist,fancyhdr,etoolbox,nicefrac,mathpazo}
\usepackage{listings}
\usepackage{tikz-cd}
\usepackage{parskip,xfrac}
\usepackage{datetime}
\usepackage{hyperref}
\usepackage{cleveref}
\usepackage{etoolbox}
\usepackage[top=27mm,bottom=27mm,left=25mm,right=25mm]{geometry}
\hypersetup{
    colorlinks=true,
    citecolor=red,       
    linkcolor=darkblue,  
    urlcolor=red         
}
\usepackage[T1]{fontenc}
\usepackage[default,scale=1.5]{comicneue} 
\usepackage[italic]{mathastext} 
\DeclareFontSeriesDefault[rm]{md}{bf}
\DeclareFontSeriesDefault[sf]{md}{bf}

\AtBeginDocument{\bfseries\boldmath}

\let\oldnormalfont\normalfont
\renewcommand{\normalfont}{\oldnormalfont\bfseries}
\boldmath
\usepackage[normalem]{ulem}     
\newcommand{\ultrabold}{\fontseries{b}\selectfont}
\definecolor{darkblue}{rgb}{0.0, 0.0, 0.55} 

\makeatletter
\def\section{\@startsection{section}{1}%
  \z@{.7\linespacing\@plus\linespacing}{.5\linespacing}%
  {\normalfont\Large\bfseries\itshape\ultrabold\color{darkblue}}}
  
\def\subsection{\@startsection{subsection}{2}%
  \z@{.5\linespacing\@plus.7\linespacing}{-.5em}%
  {\normalfont\large\bfseries\itshape\ultrabold\color{darkblue}}}
\makeatother

\makeatletter
\let\oldsection\section
\renewcommand{\section}{\@ifstar{\@starredsection}{\@unstarredsection}}
\newcommand{\@starredsection}[1]{\oldsection*{\uline{#1}}}
\newcommand{\@unstarredsection}[1]{\oldsection{\uline{#1}}}

\let\oldsubsection\subsection
\renewcommand{\subsection}{\@ifstar{\@starredsubsection}{\@unstarredsubsection}}
\newcommand{\@starredsubsection}[1]{\oldsubsection*{\uline{#1}}}
\newcommand{\@unstarredsubsection}[1]{\oldsubsection{\uline{#1}}}
\makeatother

\makeatletter
\newtheoremstyle{underlinedthm}
  {6pt}
  {6pt}
  {\itshape}
  {}
  {\bfseries\itshape\ultrabold\color{darkblue}}
  {.}
  {.5em}
  {\uline{\thmname{#1}\thmnumber{ #2}\thmnote{ (#3)}}}

\newtheoremstyle{underlineddef}
  {6pt}
  {6pt}
  {\normalfont}
  {}
  {\bfseries\itshape\ultrabold\color{darkblue}}
  {.}
  {.5em}
  {\uline{\thmname{#1}\thmnumber{ #2}\thmnote{ (#3)}}}
\makeatother

\theoremstyle{underlinedthm} 
\makeatletter
\def\@settitle{\begin{center}%
  \baselineskip14\p@\relax
  \huge\bfseries\itshape\ultrabold\color{darkblue}%
  \expandafter\uline\expandafter{\@title}
  \end{center}%
}
\makeatother

\newtheorem{theorem}{Theorem}[]
\newtheorem{definition}{Definition}[]
\newtheorem{discussion}{Discussion}[]

\newtheorem{lemma}{Lemma}

\newtheorem{notation}[definition]{Notation}

\theoremstyle{remark}

\newcommand{\disc}{\mathfrak{disc}}

\newcommand{\Z}{\mathbb{Z}}

\newcommand{\D}{\mathfrak{O}}

\newcommand{\Res}{\mathfrak{Res}}
\newcommand{\ksh}{\mathfrak{Ksh}}

\newcommand{\J}{\mathcal{J}}
\begin{document}

\title[Ekedahl for discriminant]{On the Ekedahl sieve for the singular locus of the discriminant polynomial.}
\author[G.D.Patil]{Gaurav Digambar Patil}
\address{LMSI address}
\email{gaurav231992@gmail.com}
\href{https://orcid.org/0009-0008-7073-0160}{orcid.org/0009-0008-7073-0160}
\date{\today}

\maketitle

\begin{abstract}
    The Ekedahl sieve is a powerful tool for enumerating arithmetic objects, but traditional formulations relying on inductive steps often yield suboptimal bounds when applied to highly skew boxes. This limitation is particularly restrictive when introducing large modular conditions that compete with the tail-end variables of a binary form. In this paper, we develop a specialized variant of the Ekedahl sieve tailored to the singular locus of the discriminant polynomial. By exploiting the specific non-degeneracy properties of the discriminant outside its two most extreme coefficients, we bypass the standard inductive framework, reducing the sieve to a highly efficient two-step process in some cases, and a one step process in others. We establish robust generic tail-end estimates as well as squarefree, power-saving bounds that seamlessly incorporate external modular conditions. This optimization maximizes the permissible range of the sieving modulus, yielding improved error terms for the enumeration of bounded squarefree values of certain polynomials and providing the foundational geometric sieve estimates required for the weighted enumeration of number fields by discriminant.
\end{abstract}

\section{Introduction}
The Ekedahl sieve has proven to be an effective tool for the enumeration of arithmetic objects. In its original formulation, the Ekedahl sieve functions as a type of infinite Chinese Remainder Theorem, enabling the enumeration of points within a uniformly expanding affine box subject to infinite local conditions—specifically, the avoidance of the modulo $p$ reduction of a fixed variety of co-dimension two or greater over $\mathbb{Z}$ for every prime $p$. Poonen (\cite{PoonenEkedahl}) extended this framework to accommodate skew boxes, where the sides need not expand uniformly.

A critical component of these sieving arguments is the enumeration of the tail end: upper bounding the set of points satisfying the local conditions for large primes $p > M$. Historically, both Ekedahl's and Poonen's (\cite{PoonenEkedahl}) approaches yielded relatively modest error savings for the tail end because applying modulo $\prod_{p<M} p$ conditions required $\exp(M)$ to be bounded by the smallest side of the box. Bhargava later refined this tail-end enumeration by counting points modulo squarefree values. This modification achieved power-saving error terms for uniformly expanding boxes, effectively approximating the inclusion-exclusion process with terms of exponentially smaller size.

We add modular conditions say $\bmod m$ to the error terms which we demonstrate for the singular locus of the discriminant form. Varying $m$ in as large a range as possible can allow us to include more arithmetic objects of certain type in our counts. This will allow application of the sieve for a far larger set of modular conditions than limited by traditional power saving. However, maximizing the range of the modulus $m$ for certain arithmetic applications requires highly skew boxes where our modulus competes with tail end for space in the range of certain variables. This is particularly true when $m$ and the discriminant aggressively compete in the range of the final two coefficients of a binary form. In this highly skew regime, traditional inductive approaches to the Ekedahl sieve struggle to yield optimal bounds.

In this paper, we demonstrate that the inductive steps inherent in Ekedahl's and Poonen's proofs (\cite{PoonenEkedahl}) can be largely bypassed by exploiting the specific algebraic structure of the discriminant polynomial (this is true in almost all cases, for obvious reasons). The traditional inductive approach is only strictly necessary when the forms are likely to degenerate independently of the choice of some variables. By carefully analyzing the structure of the discriminant polynomial, we show that it is highly non degenerate outside of the first two and final two variables. This property allows us to reduce the sieve to a two-step process for binary forms and a one step process for monic polynomials. 

Consequently, we can achieve savings as large as the second-largest variable range in the term that does not contain the tail-end limiter. While the limiter term typically dominates this saving, isolating the tail-end term into length-$p$ boxes allows us to seamlessly introduce external modular conditions. This optimization not only yields improved bounds for highly skew boxes but also provides the foundational geometric sieve estimates necessary for more complex inclusion-exclusion enumerations in subsequent work.

The main purpose of this setup is to count how often a specific form takes bounded squarefree values, and to facilitate the weighted enumeration (yielding lower bounds) of number fields by discriminant. Furthermore, this framework will serve as the basis for a subsequent paper focusing exclusively on lower bounds without the requirement for weights
\section{Preliminaries: Binary Forms and the Discriminant Polynomial}
We begin by defining the discriminant polynomial as a polynomial in the coefficients of a binary form and describe the structure of such a form by looking at it modulo some products and greatest common divisors of the final two coefficients of a generic binary form.
\begin{notation}
Let $f(X,Y):= a_0X^n+a_{1}X^{n-1}Y +\cdots +a_nY^n$
denote a generic binary form of degree $n$.
Let $\disc(a_0,a_{1},\cdots, a_n):=\disc(f)$ denote the discriminant of $f$ as a polynomial in the coefficients of $f.$ Let $\Res$ denote the resultant function on pairs of binary forms. The $\Res$ when restricted to forms of fixed degrees will thus be a fixed polynomial in the coefficients of pair of forms it is acting upon.
\end{notation}
The following are well known lemmata.
\begin{lemma}\label{DiscirminantResultantRelationship} The functions $\disc$ and $\Res$ satisfy
\[
(-1)^{\frac{n(n-1)}{2}}a_0a_n\disc(f)=a_n \Res(f,\frac{\partial f}{\partial X} )=a_0 \Res(f, \frac{\partial f}{\partial Y}).
\]
\end{lemma}

\begin{notation}\label{SylvesterforDisc}
    We denote the Sylvester matrix for $(f,\frac{\partial f}{\partial X})$ by $Syl_{f,x}.$ When written explicitly, $Syl_{f,X}$ is
\begin{equation*}
    \begin{bmatrix}
    a_0    & a_1      & a_2      & \cdots & a_{n-1}  & a_n      & 0       & \cdots & 0      & 0      & 0      \\
    0      & a_0      & a_1      & \cdots & a_{n-2}  & a_{n-1}  & a_n     & \cdots & 0      & 0      & 0      \\ 
    0      &   0      & a_0      & \cdots & a_{n-3}  & a_{n-2}  & a_{n-1} & \cdots & 0      & 0      & 0      \\
    \vdots & \vdots   & \vdots   & \ddots & \vdots   & \vdots   & \vdots  & \ddots & \vdots & \vdots & \vdots \\ 
    0      &   0      & 0        & \cdots & a_3      & a_4      & a_5     & \cdots & a_n    & 0      & 0      \\
    0      &   0      & 0        & \cdots & a_2      & a_{3}    & a_4     & \cdots & a_{n-1}& a_n    & 0      \\ 
    0      &   0      & 0        & \cdots & a_1      & a_{2}    & a_{3}   & \cdots & a_{n-2}& a_{n-1}& a_n      \\
    na_0   & (n-1)a_1 & (n-2)a_2 & \cdots & a_{n-1}  & 0        &     0   & \cdots & 0      & 0      & 0      \\
    0      & na_0     & (n-1)a_1 & \cdots & 2a_{n-2} & a_{n-1}  &     0   & \cdots & 0      & 0      & 0      \\
    0      &   0      & na_0     & \cdots & 3a_{n-3} & 2a_{n-2} & a_{n-1} & \cdots & 0      & 0      & 0      \\
    \vdots & \vdots   & \vdots   &\ddots  & \vdots   & \vdots   & \vdots  & \ddots & \vdots & \vdots & \vdots \\
    0      &   0      & 0        & \cdots &(n-2)a_2  & (n-3)a_3 & \cdots  & \cdots & a_{n-1}& 0      & 0      \\
    0      &   0      & 0        & \cdots &(n-1)a_{1}& (n-2)a_2 & (n-3)a_3& \cdots &2a_{n-2}& a_{n-1}& 0      \\ 
    0      &   0      & 0        & \cdots & na_0     & (n-1)a_1 & (n-2)a_2& \cdots &3a_{n-3}&2a_{n-2}& a_{n-1}      
\end{bmatrix}.
\end{equation*}
We define $Syl_{f,y}$ similarly.
\end{notation}
The following is also a well known lemma.
\begin{lemma}\label{DetofSylvesterisRes}
    We have
    $\det(Syl_{f,X})=\Res(f,\frac{\partial f}{\partial X})$
\end{lemma}

The following lemma is easy to observe from \cref{DetofSylvesterisRes} and \cref{SylvesterforDisc} and \cref{DiscirminantResultantRelationship}.

\begin{lemma}[Disc structure final two]\label{WeaklyDivisibleStructureLemma} 
We can write $\disc(a_0,a_1,\cdots,a_n)$ in the following form:
    \[
    \disc(a_0,a_1,\cdots,a_n)=4a_n a_{n-2} \J_1 +a_n a_{n-1} \J_2 +a_{n-1}^2\J_3  +a_n^2 \J_4.
    \]
    where $\J_1,\J_2,\J_3,\J_4\in \Z[a_0,a_1,\cdots, a_n], \J_1=a_{n-1}^2\disc(\frac{f-a_ny^n-a_{n-1}xy^{n-1}}{x^2})$ and
    $\J_3=\disc(\frac{f(X,Y)-a_nY^2}{Y})$.
\end{lemma}
\begin{proof}
    The proof follows from expanding $\det(Syl_{f,X})$ along the last two columns as 
    \begin{equation}\label{str}
         2a_n a_{n-2}A_1+a_{n-1}^2 A_2 +a_n a_{n-1} A_3 + a_n^2 A_4
    \end{equation}
    and then considering $A_2$ modulo $a_n$ and $A_1$ modulo $(a_n,a_{n-1}).$
    We carefully expand the determinant of the Sylvester Matrix along the last two columns. We can clearly see we can write the determinant as 
    \begin{equation}\label{str}
        2a_n a_{n-2}A_1+a_{n-1}^2 A_2 +a_n a_{n-1} A_3 + a_n^2 A_4.
    \end{equation}
    We will replace $A_1$ with $\J_1\in \Z[a_0,a_1,\cdots,a_{n-2}]$ and as the components of $A_1$ containing $a_{n-1}$ and $a_{n}$ can be pushed inside $A_2,A_3,A_4$. One can then replace $A_3$ with $\J_3\in \Z[a_0,a_1,\cdots,a_{n-1}]$ as components of $A_2$ containing $a_n$ can be pushed inside $A_2$ and $A_4.$
    
    Then, $\J_1$ can be seen as the determinant of (i.e. the minor corresponding to $a_n\cdot (2a_{n-2})$ where $a_n$ and $a_{n-1}$ are set to zero looks as follows)
    \[
    \begin{bmatrix}
    a_0    & a_1      & a_2      & \cdots & 0        & 0        & 0       & \cdots & 0      \\
    0      & a_0      & a_1      & \cdots & a_{n-2}  & 0        & 0       & \cdots & 0      \\ 
    0      &   0      & a_0      & \cdots & a_{n-3}  & a_{n-2}  & 0       & \cdots & 0      \\
    \vdots & \vdots   & \vdots   & \ddots & \vdots   & \vdots   & \vdots  & \ddots & \vdots \\ 
    0      &   0      & 0        & \cdots & a_3      & a_4      & a_5     & \cdots & 0      \\
    0      &   0      & 0        & \cdots & a_2      & a_{3}    & a_4     & \cdots & 0      \\
    na_0   & (n-1)a_1 & (n-2)a_2 & \cdots & 0        & 0        &     0   & \cdots & 0      \\
    0      & na_0     & (n-1)a_1 & \cdots & 2a_{n-2} & 0        &     0   & \cdots & 0      \\
    0      &   0      & na_0     & \cdots & 3a_{n-3} & 2a_{n-2} &     0   & \cdots & 0      \\
    \vdots & \vdots   & \vdots   &\ddots  & \vdots   & \vdots   & \vdots  & \ddots & \vdots \\
    0      &   0      & 0        & \cdots &(n-2)a_2  & (n-3)a_3 & \cdots  & \cdots & 0      \\
    0      &   0      & 0        & \cdots &(n-1)a_{1}& (n-2)a_2 & (n-3)a_3& \cdots &2a_{n-2}          
    \end{bmatrix}.
    \]
    Now doing operations $R_{n}-R_1, R_{n+1}-R_2, R_{n+1} -R_3 ,\cdots, R_{2n-2}-R_{n-2}$ we get the following matrix with the same determinant.
    \[
    \begin{bmatrix}
    a_0     & a_1      & a_2      & \cdots & 0        & 0        & 0       & \cdots & 0      \\
    0       & a_0      & a_1      & \cdots & a_{n-2}  & 0        & 0       & \cdots & 0      \\ 
    0       &   0      & a_0      & \cdots & a_{n-3}  & a_{n-2}  & 0       & \cdots & 0      \\
    \vdots  & \vdots   & \vdots   & \ddots & \vdots   & \vdots   & \vdots  & \ddots & \vdots \\ 
    0       &   0      & 0        & \cdots & a_3      & a_4      & a_5     & \cdots & 0      \\
    0       &   0      & 0        & \cdots & a_2      & a_{3}    & a_4     & \cdots & 0      \\
    (n-1)a_0& (n-2)a_1 & (n-3)a_2 & \cdots & 0        & 0        &     0   & \cdots & 0      \\
    0      & (n-1)a_0 & (n-2)a_1 & \cdots & a_{n-2}   & 0        &     0   & \cdots & 0      \\
    0      &   0      & (n-1)a_0& \cdots & 2a_{n-3}   & a_{n-2}  &     0   & \cdots & 0      \\
    \vdots & \vdots   & \vdots   &\ddots  & \vdots    & \vdots   & \vdots  & \ddots & \vdots \\
    0      &   0      & 0        & \cdots &(n-3)a_2   & (n-4)a_3 & \cdots  & \cdots & 0      \\
    0      &   0      & 0        & \cdots &(n-1)a_{1} & (n-2)a_2 & (n-3)a_3& \cdots &2a_{n-2}          
    \end{bmatrix}.
    \]
    Expanding the determinant along the last column we see that the determinant of the above matrix is twice of the discriminant of the form
    \[
    \frac{f(X,Y)-a_nY^n-a_{n-1}XY^{n-1}}{Y}=X(a_0X^{n-2}+a_1X^{n-3}Y+\cdots+a_{n-2}Y^{n-2}).
    \]
    This can be easily observed by expanding the determinant of the above matrix and that of the Sylvester Matrix along the final column.
    
    Clearly, since $a_{n-2}$ is the Resultant of $X$ and $a_0X^{n-2}+a_1X^{n-3}Y+\cdots+a_{n-2}Y^{n-2},$ we get
    \[
    \disc(\frac{f(X,Y)-a_nY^n-a_{n-1}XY^{n-1}}{Y})= a_{n-2}^2\cdot \disc(a_0X^{n-2}+a_1X^{n-3}Y+\cdots+a_{n-2}Y^{n-2}).
    \]
    Substituting back into \cref{str} we get the appropriate structure for $\J_1$.

    To get the structure of $\Delta_3$ we see that 
    \[
    a_{n-1}^2\J_3=\disc(f(X,Y)-a_nY^n)=a_{n-1}^2\disc(\frac{f(X,Y)-a_nY^2}{Y}).
    \]
    by definition.
\end{proof}

Finally, since we know that if $f(X,Y)=a_0\prod_{i=1}^n(X-\delta_iY)$ then $\disc(f)=a_0^{2n-2} \prod_{1\le i<j\le n}(\delta_i-\delta_j)^2$, which is a symmetric homogeneous polynomial in $\delta_i$. Thus, by the Fundamental theorem of symmetric polynomials $\disc$ can be written in terms of elementary symmetric polynomials in $\delta_i$ in an order preserving way. In other words, we get the following lemma.
\begin{lemma}\label{homogeniety of order}
    If $a_0^{\lambda_0}a_1^{\lambda_1}\cdots a_n^{\lambda_n}$ is supported in $\disc,$ then
    \[
    \sum_{i=0}^n i\lambda_i=n(n-1).
    \]
\end{lemma}
\section{Structure of discriminant polynomial and its singular locus.}
\begin{definition}
    We define $\Delta_i(X_0,X_1,\cdots, X_{n-1})$ and $\mathfrak{J}(X_0,X_1,\cdots, X_{n-1})$ by 
    \[
    \disc_X(X_0 X^n +X_1 X^{n-1}+\cdots+X_n)= \Delta_{n-1}X_n^{n-1} + \sum_{i=1}^{n-2} X_n^{i} \Delta_i + \mathfrak{J}.
\]
\end{definition}
\begin{lemma}\label{Structuredisc}
We have,
\begin{align*}
    \Delta_{n-1}= n^n X_0^{n-1} \\
    \mathfrak{J}= X_{n-1}^2\disc_X(X_0 X^{n-1} +X_1 X^{n-2}+\cdots +X_{n-1})\\
    \Delta_{n-2}\equiv -(n-1)^{n-1}X_1^n \bmod X_0\\
\end{align*}
\end{lemma}
\begin{proof}
    Obvious from the understanding that 
    $X_0\disc_X(X_0 X^n +X_1 X^{n-1}+\cdots+X_n)$ is the resultant of $X_0 X^n +X_1 X^{n-1}+\cdots+X_n$ and $\frac{\partial}{\partial X}(X_0 X^n +X_1 X^{n-1}+\cdots+X_n)$ and the consideration of this resultant as the determinant of the Sylvester matrix (see \cref{SylvesterforDisc} and \cref{WeaklyDivisibleStructureLemma}). The homogeneous degree claim follows directly from looking at the discriminant as $\frac{\det(Syl_{f,x})}{a_0^2}.$
\end{proof}
Furthermore, if  $(X_0,X_1,\cdots, X_{n-1})^\lambda$ is supported in $\Delta_i$ then \cref{homogeniety of order} gives, 
\[
    \sum_{j=0}^{n-1}j\lambda_j =(n-i-1)n.
\]
Further-furthermore, as a consequence we have $\Delta_i$ has degree at most $n-i-1$ in the variable $X_{n-1}$ for $i\ge 1$.
\begin{definition}\label{singulardisc}
    We define the polynomial $\ksh$ by 
    \[
    \ksh(X_0,X_1,\cdots, X_{n-1}):=\disc_{X_n}(\disc_X(X_0 X^n +X_1 X^{n-1}+\cdots+X_n)).
\]
\end{definition}
\begin{discussion}
    We wish to study the degeneracy/constancy of $\ksh$ (up to $\gcd(X_0,X_1)$) as a polynomial in $X_{n-1}$, which we will do by studying its degree in the variable $X_{n-1}$.
To that end we define the following.
\end{discussion}
\begin{definition}
    We define $\D_i(Y_0,Y_1,\cdots, Y_{n-2})$ by 
    \[
    \disc_X(Y_0 X^{n-1} +Y_1 X^{n-2} +\cdots+Y_{n-1})=
     \D_{n-2} Y_{n-1}^{n-2} 
    + \sum_{i=0}^{n-3}\D_i  Y_{n-1}^{i} 
    \]
\end{definition}
\begin{discussion}
    Clearly, we have 
    \[
    \ksh=\D_{n-2}(\Delta_{n-1},\Delta_{n-2},\dots,\Delta_1)\mathfrak{J}^{n-2} 
    + \sum_{i=0}^{n-3} \D_i(\Delta_{n-1},\Delta_{n-2},\dots,\Delta_1)  \mathfrak{J}^{i}.
    \] 
    Now if $(Y_0,Y_1,\cdots,Y_{n-2})^\delta$ is supported in $\D_i$, then $\sum_{j=1}^{n-2}j\delta_j=(n-2-i)(n-1).$ Since $\Delta_{n-1-j}$ has degree at most $j$ in $X_{n-1}$, the degree in $X_{n-1}$ of this monomial will be at most 
    \[
    \sum_{j=1}^{n-2}j\delta_j=(n-2-i)(n-1).
    \]
    Since $\mathfrak{J}$ has degree $n$ in $X_{n-1}$, it follows that $\D_i(\Delta_{n-1},\Delta_{n-2},\dots,\Delta_1)  \mathfrak{J}^{i}$ has degree at most $(n-2-i)(n-1) +ni=(n-2)(n-1)+i$ in $X_{n-1}$.
    
    On the other hand, from \cref{Structuredisc} we note that 
    \[
    \mathfrak{J}=(n-1)^{n-1}(n^nX_0^{n-1})X_{n-1}^n+(-(n-2)^{n-2}X_1^{n-1} +X_0Jaws)X_{n-1}^{n-1} + \textit{ lower degree in $X_{n-1}$ terms} 
    \]
    for some $Jaws\in \Z[X_0,X_1,\cdots,X_{n-2}]$.
    On the third hand, the lead term looks like 
    \[
    (n-1)^{n-1}(n^nX_0^{n-1})^{n-2}( \mathfrak{J})^{n-2}.
    \]
    Binomial expansion after substitution gives 
    \[
    \ksh= (n-1)^{(n-1)^2}n^{2n(n-2)}X_0^{2(n-1)(n-2)}\cdot X_{n-1}^{n(n-2)} +\textit{terms with lower degree in $X_{n-1}$.}
    \]
    Now consider $\ksh \bmod X_0$; it is clear that there will be no contribution from the leading term.
    We now turn our attention to $\D_{n-3}(\Delta_{n-1},\cdots, \Delta_1)\mathfrak{J}^{n-3}$. Again applying \cref{Structuredisc} tells us 
    \[
    \D_{n-3}(\Delta_{n-1},\cdots, \Delta_1)\mathfrak{J}^{n-3}\equiv -(n-2)^{n-2}\Delta_{n-2}^{n-1}\mathfrak{J}^{n-3}\bmod \Delta_{n-1}
    \]
    Since $X_0|\Delta_{n-1}$ and $\Delta_{n-2}\equiv -(n-1)^{n-1}X_1^n \bmod X_0$, it follows that 
    \[
    \D_{n-3}(\Delta_{n-1},\cdots, \Delta_1)\mathfrak{J}^{n-3}\equiv -(n-2)^{n-2}(-(n-1)^{n-1}X_1^n)^{n-1}\mathfrak{J}^{n-3} \bmod X_0.
    \]
    Applying \cref{Structuredisc} to study $\mathfrak{J}\bmod X_0$, we get 
    \[
    \mathfrak{J}\equiv X_{n-1}^2 X_1^2\disc(X_1X^{n-2}+\cdots +X_{n-1})\]
    \[
    \equiv (n-2)^{n-2}(X_1X_{n-1})^{n-1} +(\textit{lower degree terms in $X_{n-1}$}) \bmod X_0
    \]
    Thus,
    \[
    \mathfrak{J}^{n-3}\equiv (n-2)^{(n-2)(n-3)}(X_1X_{n-1})^{(n-1)(n-3)}+(\textit{lower degree terms in $X_{n-1}$}) \bmod X_0.
    \]
    \[
    \D_{n-3}(\Delta_{n-1},\cdots,\Delta_1)\mathfrak{J}^{n-3}\equiv (-1)^n(n-2)^{(n-2)^2}(n-1)^{(n-1)^2}X_1^{(2n-3)(n-1)}X_{n-1}^{(n-1)(n-3)} \bmod X_0.
    \]
    Noting that the leading term is $0\bmod X_0$, and the later terms have strictly lower degree in $X_{n-1}$ as argued above, we can conclude the following.
\end{discussion}
\begin{lemma}\label{Structuresingulardisc} If $p\nmid \gcd(a_0,a_1)$ and $p\nmid2(n-1)$, then for any choice of $a_2,a_3,\cdots,a_{n-2}$ we have $\ksh(a_0,a_1,\cdots,a_{n-2},T)\bmod p$ is never degenerate or rather has degree either $(n-1)(n-3)$ (if $p|a_0$) or $n(n-2)$ (if $p\nmid a_0$) in $T$.
Furthermore, for any choice of $a_{n-1}$, $\disc(a_0,a_1,\cdots,a_{n-1},S) \bmod p$ is never degenerate or rather has degree either $n-1$ (if $p|a_0$) or $n-2$ (if $p\nmid a_0$).
\end{lemma}
\section{Generic Tail end Ekedahl estimates.}
Note that \cref{Structuresingulardisc}, tells of we can be far more free with the final two variables as they never degenerate. This eliminates the dependence on induction in original Poonen's/Ekedahl's proof (\cite{PoonenEkedahl}), allowing for a much larger saving potentially in more skew boxes. In our case, we will further introduce useful and modular conditions while tracking how the size fo the modulus affects the estimate. The modular conditions might make no sense here but will be used in future papers. Either way it is still proof of concept. We thus get the following theorem (by following the exact route of Ekedahl, but by incorporating the mild properties of discriminant shown in the previous subsection).
\begin{theorem}[Disc Ekedahl Modular]\label{tail end general}
    Given $a_0,a_1,\cdots,a_{n-2}\in \Z$ and $l\in \sfrac{\Z}{m\Z}$ satisfying $p|\gcd(a_0,a_1) $ $\implies p<M \textit{ or } p|m,$, then the number of $(s,t)$ with $|s|\le A$ and $|t|\le B$ such that $f=a_0X^n+a_1X^{n-1}Y+\cdots+sXY^{n-1}+tY^n$, satisfies $m^2|f(l,1)$ and $m|\frac{\partial f}{\partial X}(l,1)$ and there exists $p>M$ (prime) and $p\nmid m$ such that 
    \[
    p\mid \gcd(\ksh(a_0,a_1,\cdots,a_{n-2},s), \disc(a_0,a_1,\cdots, a_{n-2},s,t))
    \] 
    is 
    \[
    O_n(1)\left(\frac{T}{R(\log R+\log T) }+\frac{1}{M\log M}+\frac{m}{A}+\frac{m^3}{AB}\right)\cdot \frac{AB}{m^3}
    \]
    where the $O$ is independent of the choice of $a_i$ and $T= \log \max(\{|a_i|:0\le i\le n-2\}\cup\{A\})$ and $R=\min(\frac{A}{m},\frac{B}{m^2})$ and $m$.
\end{theorem}
\begin{proof}
    We count the number of $(s,t,p)$ such that $f$ is weakly divisible by $m$ and strongly divisible by $p$ as usual by dividing it into two main sets and a degenerate third set.
    For each choice of $p$, the number of $\bmod p$ solutions $(s\bmod p, t\bmod p)$ such that $f$ is strongly divisible by $p$ (which is a condition weaker than $p|\gcd(\ksh(f),\disc(f))$) is $O_n(1)$ from \cref{Structuresingulardisc}.
\begin{itemize}
        \item The first set is defined by $S_1:=\{(s,t,p): M<p\le\min(\frac{A}{m},\frac{B}{m^2})\}.$ Clearly, the number of elements in $S_1$ can be approximated by the number of $p$-boxes times the number of solutions to $\ksh$ and $\disc$ being divisible by $p$. This is at most $O_n(1)\cdot(\lfloor\frac{A}{mp}\rfloor +1)(\lfloor\frac{B}{m^2p}\rfloor + 1)$. Summing over all primes in the range gives 
        \[
        O_n(1)\frac{AB}{m^3M\log M}.
        \]
        \item The second set is defined by $S_2:=\{(s,t,p): \min(\frac{A}{m},\frac{B}{m^2})<p\}$. Let $R=\min(\frac{A}{m},\frac{B}{m^2})$. 
        For every choice of $s$, the number of distinct primes $p>R$ (and $p\nmid m$) dividing $\ksh(a_0,a_1,\cdots,a_{n-2},s)$ (where $\ksh\neq 0$) is at most $\frac{c\log R}{\log c +\log R}$ if $\ksh=R^c$. Since $\ksh=T^{O_n(1)}$, it follows that the number of such primes is at most $O_{n}(1)\frac{T}{\log R+\log T}$. Now for the final coordinate we have at most $O_n(1)\lfloor\frac{B}{m^2R} +1\rfloor$ possibilities for the last coordinate. Thus, the total number of elements in $S_2$ is at most 
        \[
        O_n(1)\frac{T}{R(\log R+\log T)}\frac{AB}{m^3}.
        \]
        \item The third set is defined by $\ksh=0$, which is at most 
        \[
        O_n(1)\left(\frac{B}{m^2}+1\right).
        \]
    \end{itemize}
\end{proof}
\begin{discussion}
    One can follow the exact same procedure but count $(s,t,\varsigma)$ where $\varsigma$ is squarefree, greater than $M$ and co-prime to $m$ the three pieces of the bounds above obviously change to 
\begin{itemize}
    \item 
    \[
    (O_n(1))^{\omega(\varsigma)}\frac{AB}{m^3M}
    \]
    \item (Here we just consider all possible squarefree divisors of $\ksh$ as an upper bound instead of only those that are greater than $R$)
    \[
    2^{2\omega(T^{O_n(1)})}\frac{AB}{m^3R}
    \]
    \item The third set remains the same.
    \[
     O_n(1)\left(\frac{B}{m^2}+1\right).
    \]
\end{itemize}
Thus an alternative version that might allow for power saving in the vein of Bhargava's idea in \cite{bhargava2014geometricsievedensitysquarefree} gives us,
\begin{theorem}\label{squarefree tail end general}
    Given $a_0,a_1,\cdots,a_{n-2}\in \Z$ and $l\in \sfrac{\Z}{m\Z}$ satisfying $rad{(\gcd(a_0,a_1))}|m$, then the number of $(s,t)$ with $|s|\le A$ and $|t|\le B$ such that $f=a_0X^n+a_1X^{n-1}Y+\cdots+sXY^{n-1}+tY^n$, satisfies $m^2|f(l,1)$ and $m|\frac{\partial f}{\partial X}(l,1)$ and there exists $\varsigma >M$ (squarefree) and $\gcd(\varsigma,m)=1$ such that 
    \[
    \varsigma\mid \gcd(\ksh(a_0,a_1,\cdots,a_{n-2},s), \disc(a_0,a_1,\cdots, a_{n-2},s,t))
    \] 
    is 
    \[
    O_n(1)\left(\frac{T^{\frac{O_{n}(1)}{\log \log T}}}{R}+\frac{1}{M}+\frac{m}{A}+\frac{m^3}{AB}\right)\cdot \frac{AB}{m^3}
    \]
    where the $O$ is independent of the choice of $a_i$ and $T= \log \max(\{|a_i|:0\le i\le n-2\}\cup\{A\})$ and $R=\min(\frac{A}{m},\frac{B}{m^2})$ and $m$.
\end{theorem}
Taking $m=1$, $l=0$ $a_0=1$ and all possible $a_1,a_2,a_3,\cdots,a_n$ where $|a_i|\le X^{i}$, in the notation in \cite{BSW-monicsquarefree} gives us the following result:
\begin{theorem}With definition of $\mathcal{W}^{(1)}_m$ in \cite{BSW-monicsquarefree}, we get
    \[
    \#\bigcup_{\substack{\varsigma >M\\ \varsigma \textit{ squarefree} }} \{f\in \mathcal{W}^{(1)}_m:H(f)\le X\}=O(\frac{X^{\frac{n(n+1)}{2}+\frac{O_n(1)}{\log\log X}}}{M} + X^{\frac{n(n+1)}{2}+\frac{O_n(1)}{\log\log X}-(n-1)}).
    \]
\end{theorem}

This is a power-saving improvement over theorem 1.5 a in \cite{BSW-monicsquarefree}. 

\end{discussion}

Similarly, we get the following theorem.
\begin{theorem}\label{corpimalty tail end}
    \[
    |\{(a_0,a_1):|a_0|\le C, |a_1|\le D, a_0a_1\neq 0, \exists p>M: p|\gcd(a_0,a_1)\}|\le 4\left(\frac{1}{M\log M}\right)CD.
    \]
\end{theorem}
This together with \cref{tail end general} allows us to use the Ekedahl sieve on various spaces as long as the first two and final two coefficients are freely chosen. We can use this to count the number of bounded squarefree values of certain multi-variable polynomials as well as a weighted enumeration of number fields.


\bibliographystyle{abbrv}
\bibliography{main}
\end{document}